\documentclass[11pt]{article}

\usepackage{amsfonts}
\usepackage{shuffle,yfonts,amsmath,amssymb,amsthm}
\usepackage{hyperref}
\hypersetup{pdfstartview={XYZ null null 1.00}}
\usepackage{pdfpages}

\newcommand\nc{\newcommand}
 \nc{\binn}{{\binom{2n}{n}}}

 \allowdisplaybreaks

\catcode`!=11
\let\!int\int \def\int{\displaystyle\!int}
\let\!lim\lim \def\lim{\displaystyle\!lim}
\let\!sum\sum \def\sum{\displaystyle\!sum}
\let\!sup\sup \def\sup{\displaystyle\!sup}
\let\!inf\inf \def\inf{\displaystyle\!inf}
\let\!cap\cap \def\cap{\displaystyle\!cap}
\let\!max\max \def\max{\displaystyle\!max}
\let\!min\min \def\min{\displaystyle\!min}
\let\!frac\frac \def\frac{\displaystyle\!frac}
\catcode`!=12

\let\oldsection\section
\renewcommand\section{\setcounter{equation}{0}\oldsection}

\allowdisplaybreaks

\DeclareMathOperator{\Li}{Li}

 \nc{\gam}{{\gamma}}
 \nc{\gG}{{\Gamma}}
 \nc{\vep}{{\varepsilon}}
 \nc{\gs}{{\sigma}}
 \nc{\gth}{{\theta}}
 \nc{\gS}{{\Sigma}}
 \nc{\gf}{{\varphi}}
 \nc{\gk}{{\kappa}}
 \nc{\gz}{{\zeta}}
 \nc{\tgz}{{\tilde{\zeta}}}
 \nc{\gO}{{\Omega}}
 \nc{\sif}{{\mathcal S}}
 \nc{\gt}{{\tau}}

\def\N{\mathbb{N}}

\def\ze{\zeta}

\theoremstyle{plain}
\newtheorem{thm}{Theorem}[section]
\newtheorem{lem}[thm]{Lemma}

\newtheorem{con}[thm]{Conjecture}

\theoremstyle{definition}

\newtheorem{re}[thm]{Remark}

\begin{document}
\title{\bf Sun's Three Conjectures on Ap\'{e}ry-like Sums Involving Harmonic Numbers}
\author{
{Ce Xu${}^{a,}$\thanks{Email: cexu2020@ahnu.edu.cn, first author, ORCID 0000-0002-0059-7420.}\ and Jianqiang Zhao${}^{b,}$\thanks{Email: zhaoj@ihes.fr, corresponding author, ORCID 0000-0003-1407-4230.}}\\[1mm]
\small a. School of Mathematics and Statistics, Anhui Normal University, Wuhu 241002, PRC\\
\small b. Department of Mathematics, The Bishop's School, La Jolla, CA 92037, USA}

\date{}
\maketitle

\noindent{\bf Abstract.} In this paper, we will give another proof of Zhi-Wei Sun's three conjectures
on Ap\'{e}ry-like sums involving harmonic numbers by proving some identities among special
values of multiple polylogarithms.

\medskip



\section{Introduction}

In his new book \cite{Sun2021}, Prof. Zhi-Wei Sun listed 820 mathematical conjectures, including ten concerning the Ap\'{e}ry-like sums involving harmonic numbers, some of which had appeared in his previous paper \cite{Sun2015}. In this paper, we will prove the following three conjectures by using a few results of Akhilesh \cite{Akhilesh} and Davydychev--Kalmykov \cite{DavydychevDe2004}, together with Au's package \cite{Au2020,Au2021}.

\begin{con} \emph{(\cite{Sun2015,Sun2021})} We have
\begin{align*}
&\sum_{n=1}^\infty \frac{(-1)^{n-1}}{n^3 \binom{2n}{n}}\left(10H_{n}-\frac{3}{n}\right)=\frac{\pi^4}{30},\\
&\sum_{n=1}^\infty \frac{1}{n^2 \binom{2n}{n}}\left(3H_{n-1}^2+\frac{4}{n}H_{n-1}\right)=\frac{\pi^4}{360}, \\
&\sum_{n=1}^\infty \frac{(-1)^{n-1}}{n^3 \binom{2n}{n}}\left(H_{2n}+4H_n\right)=\frac{2\pi^4}{75},
\end{align*}
where $H_n$ is the classical harmonic number defined by
\begin{equation*}
H_0:=0\quad\text{and}\quad H_n:=\sum_{k=1}^n \frac{1}{k} \quad \forall n\ge 1.
\end{equation*}
\end{con}

\begin{re} The three conjectures were first proved by Chu in \cite[Example 3.2 and 3.3]{Chu2020} and \cite[Example 3.5 and 3.7]{Chu2021}. It should be emphasized that Chu also proved more similar values in the above two papers.
\end{re}

\section{Proof of conjectures}
The Riemann zeta values $\ze(k)\ (2\le k\in \N)$ are defined by
\begin{align}
\ze(k):=\sum_{n=1}^\infty \frac1{n^k}.
\end{align}
In particular, Euler determined the explicit values of zeta values function at even integers in 1775:
\begin{align*}
\zeta(2k)=-\frac{B_{2k}}{2(2k)!}(2\pi\sqrt{-1})^{2k},
\end{align*}
where $B_n$ are Bernoulli numbers defined by the generating function
\[\frac{t}{e^t-1}=\sum_{n=0}^\infty B_n\frac{t^n}{n!}.\]
These can be regarded as special values of the polylogarithm function $\Li_k(x)$. More generally,
for any $k_1,\dotsc,k_r\in\N$, the \emph{classical multiple polylogarithm function} is defined by
\begin{align*}
\Li_{k_1,\dotsc,k_r}(x_1,\dotsc,x_r):=\sum_{n_1>n_2>\cdots>n_r>0} \frac{x_1^{n_1}\dotsm x_r^{n_r}}{n_1^{k_1}\dotsm n_r^{k_r}}
\end{align*}
which converges if $|x_1\cdots x_j|<1$ for all $1\le j\le r$. In particular, for $|x|\le 1$ and $(x,k_1)\ne(1,1)$
\begin{equation*}
\Li_{k_1,\ldots,k_r}(x):=\Li_{k_1,\ldots,k_r}(x,\underbrace{1,\ldots,1}_{r-1})=
\sum_{n_1>n_2>\cdots>n_r>0} \frac{x^{n_1}}{n_1^{k_1}\dotsm n_r^{k_r}}
\end{equation*}
is the classical \emph{single-variable} multiple polylogarithm function. In particular, if $x=1$ then $\Li_{k_1,\ldots,k_r}(1)$ become the multiple zeta values $\ze(k_1,\ldots,k_r)$, namely, $\ze(k_1,\ldots,k_r):=\Li_{k_1,\ldots,k_r}(1)$.

\begin{lem} \emph{\cite[Eq. (122)]{Akhilesh}} We have
\begin{align}
&\sum_{n=1}^\infty \frac{H_{n-1}^{(2)}}{n^2\binom{2n}{n}}=\frac5{108}\ze(4),\label{exp-h2-1}\\
&2\sum_{n=1}^\infty \frac{H_{n-1}}{n^3\binom{2n}{n}}+3\sum_{n=1}^\infty \frac{\ze_{n-1}(1,1)}{n^2\binom{2n}{n}}=\frac1{18}\ze(4)\label{exp-h113-1},
\end{align}
where
\[H_{n-1}^{(2)}:=\sum_{k=1}^{n-1}\frac{1}{k^2}\quad \text{and}\quad\ze_{n-1}(1,1):=\sum_{n>k>j>0} \frac1{kj}.\]
\end{lem}

\begin{re}
The formula \eqref{exp-h2-1} has also already appeared in Hessami Pilehroods \cite{HP2012}, page 220, between (7) and (8).
\end{re}

\begin{lem} \emph{(\cite[Example 6.16 and 6.18]{Au2021})} Let $\gf:=\frac{\sqrt{5}-1}{2}$. Then we have
\begin{align}
&\sum_{n=1}^\infty \frac{(-1)^n}{n^3\binom{2n}{n}}H_n =\frac{12}{5}\Li_3(\gf)\ln(\gf)+\frac3{20}\Li_4(\gf^{2})-\frac{12}{5}\Li_4(\gf)-\frac6{25}\ze(3)\ln(\gf)
\nonumber\\&\quad\quad\quad\quad\quad\quad\quad+\frac{13}{20}\ln^4(\gf)-\frac7{50}\pi^2\ln^2(\gf)+\frac{\pi^4}{50},\label{para-harmon-exp-ale}\\
&\sum_{n=1}^\infty \frac{(-1)^n}{n^4\binom{2n}{n}}=8\Li_3(\gf)\ln(\gf)+\frac1{2}\Li_4(\gf^{2})-8\Li_4(\gf)-\frac4{5}\ze(3)\ln(\gf)
\nonumber\\&\quad\quad\quad\quad\quad\quad\quad+\frac{13}{6}\ln^4(\gf)-\frac7{15}\pi^2\ln^2(\gf)+\frac{7\pi^4}{90}.\label{para-harmon-exp-alt}
\end{align}
\end{lem}

\begin{lem} \emph{(\cite[(3.12)-(3.13)]{DavydychevDe2004})} For $u\in (-\infty,0)\cup (4,+\infty)$, define
\begin{equation*}
y:=\frac{1-\sqrt{\tfrac{u}{u-4}}}{1+\sqrt{\tfrac{u}{u-4}}}, \qquad
H_{-1,0,0,1}(-y):=\int_0^{-y} \frac{\Li_{3}(x)}{1+x}dx.
\end{equation*}
Then we have
\begin{eqnarray}
&& \hspace*{-10mm}
\sum_{n=1}^\infty \frac{u^n}{n^3\binom{2n}{n}}H_{n-1}
 =
 4 H_{-1,0,0,1}(-y)
+ \Li_{3,1}(y^2)
- 4 \Li_{3,1}(y)
- 4 \Li_{3,1}(-y)
- 6 \Li_{4}(-y)
\nonumber \\ &&
- 2 \Li_{4}(y)
+ 4 \Li_{2,1}(-y)  \ln y
+ 4 \Li_{2,1}(y)  \ln y
- 2 \Li_{2,1}(y^2)  \ln(y)
+ 4 \Li_{3}(-y) \ln (1-y)
\nonumber \\ &&
+ 2 \Li_{3}(-y) \ln y
+ 2 \Li_{3}(y)  \ln y
- \Li_{2}(y) \ln^2 y
- 4 \Li_{2}(-y) \ln y \ln (1-y)
\nonumber \\ &&
- \tfrac{1}{3} \ln^3 y \ln (1-y)
+ \tfrac{1}{24} \ln^4 y
+ 2 \zeta(2) \Li_{2}(y)
- \tfrac{1}{2} \zeta(2) \ln^2 y
+ 2 \zeta(2) \ln y \ln (1-y)
\nonumber \\ &&
+ 6 \zeta(3) \ln (1-y)
- 3 \zeta(3) \ln y
- 4 \zeta(4) \; ,\label{AN_Sb1_3}
\\ && \hspace*{-10mm}
\sum_{j=1}^\infty \frac{u^n}{n^3\binom{2n}{n}}H_{2n-1}
 =
 4 H_{-1,0,0,1}(-y)
+ \Li_{3,1}(y^2)
- 8 \Li_{3,1}(y)
- 4 \Li_{3,1}(-y)
- 6 \Li_{4}(-y)
\nonumber \\ &&
+ 2 \Li_{4}(y)
- \left[ \Li_{2}(y) \right]^2
+ 4 \Li_{2,1}(-y)  \ln y
+ 8 \Li_{2,1}(y)  \ln y
- 2 \Li_{2,1}(y^2) \ln y
+ \tfrac{1}{48} \ln^4 y
\nonumber \\ &&
+ 4 \Li_{3}(-y) \ln (1-y)
- 4 \Li_{3}(y) \ln (1-y)
+ 2 \Li_{3}(-y) \ln y
- 4 \Li_{2}(-y) \ln y \ln (1-y)
\nonumber \\ &&
+ 2 \Li_{2}(y) \ln y \ln (1-y)
- \tfrac{1}{2} \Li_{2}(y) \ln^2 y
- \tfrac{1}{6} \ln^3 y \ln (1-y)
+ 4 \zeta(2) \ln y \ln (1-y)
\nonumber \\ &&
- \zeta(2) \ln^2 y
+ 10 \zeta(3) \ln (1-y)
- 5 \zeta(3) \ln y
+ 4 \zeta(2) \Li_{2}(y)
- \tfrac{19}{2} \zeta(4). \label{AN_Sb1_4}
\end{eqnarray}
\end{lem}

\begin{re} In \cite{DavydychevDe2004}, $\Li_{2,1}(y)$ and $\Li_{3,1}(y)$ are denoted by $S_{1,2}(y)$ and $S_{2,2}(y)$, respectively. More general, for $a,b\in\N$ and $z\in [0,1]$ the function $S_{a,b}(z)$ is defined by
\begin{align*}
S_{a,b}(z):=&\, \frac{(-1)^{a-1+b}}{(a-1)!b!} \int_0^1 \frac{\ln^{a-1}(t)\ln^b(1-zt)}{t}dt\\
=&\,  \frac{1}{(a-1)!b!} \int_0^1  \left(\int_t^1  \frac{dt}{t}\right)^{a-1}  \frac{dt}{t} \left(\int_0^t\frac{z\,dt}{1-zt}\right)^b \\
=&\,  \int_0^1   \left( \frac{dt}{t}\right)^{a} \left(\frac{z\,dt}{1-zt}\right)^b \\
=&\,  \int_0^z   \left( \frac{dt}{t}\right)^{a} \left(\frac{dt}{1-t}\right)^b =\Li_{a+1,1_{b-1}}(z),
\end{align*}
where we have used Chen's iterated integrals above to represent the single-variable multiple polylogarithm, see \cite[Ch.~2]{Zhao2016}.
\end{re}

\begin{thm} We have
\begin{align}
&\sum_{n=1}^\infty \frac{(-1)^{n-1}}{n^3 \binom{2n}{n}}\left(10H_{n}-\frac{3}{n}\right)=\frac{\pi^4}{30},\label{conj2}\\
&\sum_{n=1}^\infty \frac{1}{n^2 \binom{2n}{n}}\left(3H_{n-1}^2+\frac{4}{n}H_{n-1}\right)=\frac{\pi^4}{360}, \label{conj1}\\
&\sum_{n=1}^\infty \frac{(-1)^{n-1}}{n^3 \binom{2n}{n}}\left(H_{2n}+4H_n\right)=\frac{2\pi^4}{75}\label{conj3}.
\end{align}
\end{thm}
\begin{proof}
First, applying \eqref{para-harmon-exp-ale} and \eqref{para-harmon-exp-alt} yields equation \eqref{conj2}.
To prove \eqref{conj1}, by applying the stuffle relations (or quasi-shuffle relations, see \cite{Hoffman2000}), we get
\begin{equation*}
H^2_{n-1}=2\ze_{n-1}(1,1)+H_{n-1}^{(2)}.
\end{equation*}
Hence, the \eqref{conj1} can be rewritten as the form
\begin{align*}
&\sum_{n=1}^\infty \frac{1}{n^2 \binom{2n}{n}}\left(3H_{n-1}^2+\frac{4}{n}H_{n-1}\right)=2\left(2\sum_{n=1}^\infty \frac{H_{n-1}}{n^3\binom{2n}{n}}+3\sum_{n=1}^\infty \frac{\ze_{n-1}(1,1)}{n^2\binom{2n}{n}}\right)+3\sum_{n=1}^\infty \frac{H_{n-1}^{(2)}}{n^2\binom{2n}{n}}.
\end{align*}
Then, using \eqref{exp-h2-1} and \eqref{exp-h113-1} gives
\begin{align*}
&\sum_{n=1}^\infty \frac{1}{n^2 \binom{2n}{n}}\left(3H_{n-1}^2+\frac{4}{n}H_{n-1}\right)=\frac2{18}\ze(4)+\frac{3\cdot 5}{108}\ze(4)=\frac1{4}\ze(4)=\frac{\pi^4}{360}.
\end{align*}

Finally, we prove \eqref{conj3}. Observe that
\begin{align}\label{con1-3}
&\sum_{n=1}^\infty \frac{(-1)^{n-1}}{n^3 \binom{2n}{n}}\left(H_{2n}+4H_n\right)=\sum_{n=1}^\infty \frac{(-1)^n}{n^3\binom{2n}{n}}(H_{n-1}-H_{2n-1})-5\sum_{n=1}^\infty \frac{(-1)^n}{n^3\binom{2n}{n}}H_n+\frac1{2}\sum_{n=1}^\infty \frac{(-1)^n}{n^4\binom{2n}{n}}.
\end{align}
The difference of \eqref{AN_Sb1_3} and \eqref{AN_Sb1_4} yields
\begin{align}\label{para-harmon-exp}
\sum_{n=1}^\infty \frac{u^n}{n^3\binom{2n}{n}}(H_{n-1}-H_{2n-1})&=4\Li_{3,1}(y)-4\Li_4(y)+\Li_2^2(y)-4\Li_{2,1}(y)\ln y\nonumber\\
&\quad+2\Li_3(y)\ln y-\frac1{2}\Li_2(y)\ln^2y+4\Li_3(y)\ln(1-y)\nonumber\\
&\quad-2\Li_2(y)\ln y \ln(1-y)-\frac1{6}\ln^3 y\ln(1-y)+\frac1{48}\ln^4y\nonumber\\
&\quad-2\ze(2)\Li_2(y)+\frac1{2}\ze(2)\ln^2 y-2\ze(2)\ln y\ln (1-y)\nonumber\\
&\quad-4\ze(3)\ln(1-y)+2\ze(3)\ln y+\frac{11}{2} \ze(4).
\end{align}
Setting $u=-1$ and $y=\frac{3-\sqrt{5}}{2}=\gf^2$ in \eqref{para-harmon-exp} we find
\begin{align}\label{para-harmon-exp-1}
\sum_{n=1}^\infty \frac{(-1)^n}{n^3\binom{2n}{n}}(H_{n-1}-H_{2n-1})&=4\Li_{3,1}(\gf^2)-4\Li_4(\gf^2)+\Li_2^2(\gf^2)-8\Li_{2,1}(\gf^2)\ln(\gf)\nonumber\\
&\quad+8\Li_3(\gf^2)\ln(\gf)-6\Li_2(\gf^2)\ln^2(\gf)-\ln^4(\gf)\nonumber\\
&\quad-2\ze(2)\Li_2(\gf^2)-2\ze(2)\ln^2(\gf)+\frac{11}{2}\ze(4).
\end{align}
Combining \eqref{para-harmon-exp-ale}, \eqref{para-harmon-exp-alt}, \eqref{con1-3} and \eqref{para-harmon-exp-1},  we get
\begin{align}\label{fe-0}
&\sum_{n=1}^\infty \frac{(-1)^{n-1}}{n^3 \binom{2n}{n}}\left(H_{2n}+4H_n\right)\nonumber\\
&=-8\Li_3(\gf)\ln(\gf)-\frac9{2}\Li_4(\gf^2)+8\Li_4(\gf)+4\Li_{3,1}(\gf^2)+\Li_2^2(\gf^2)\nonumber\\
&\quad-8\Li_{2,1}(\gf^2)\ln(\gf)+8\Li_{3}(\gf^2)\ln(\gf)-6\Li_{2}(\gf^2)\ln^2(\gf)\nonumber\\
&\quad-2\ze(2)\Li_{2}(\gf^2)+\frac4{5}\ze(3)\ln(\gf)-\frac{19}{6}\ln^4(\gf)+\frac{2}{15}\pi^2\ln^2(\gf).
\end{align}
By applying Au's package \cite{Au2021} (also see Remark~\ref{re:humanProof} below), we have
\begin{align}
&\begin{aligned}
&\Li_{2,1}(\gf^2)=\text{MZIteratedIntegral}[{0, \gf^{-2}, \gf^{-2}}]=\ze(3)+\frac{\pi^2}{10}\ln(\gf)-\Li_3(\gf),\label{fe-1}
\end{aligned}\\
&\begin{aligned}
\Li_{3,1}(\gf^2)=\text{MZIteratedIntegral}[{0, 0,\gf^{-2}, \gf^{-2}}]&=\frac{\pi^4}{90}-\frac{\pi^2}{20}\ln^2(\gf)+\frac3{8}\ln^4(\gf)+\frac9{8}\Li_4(\gf^2)\\
&\quad-2\Li_4(\gf)+\frac1{5}\ze(3)\ln(\gf).\label{fe-2}
\end{aligned}
\end{align}
Further, from \cite[(1.20) and (6.13)]{Lewin1981} we have
\begin{align}
\Li_2(\gf^2)=&\, \frac{\pi^2}{15}-\ln^2(\gf),\label{fe-3} \\
\Li_3(\gf^2)=&\, \frac4{5}\ze(3)-\frac2{3}\ln^3(\gf)+\frac2{15}\pi^2\ln(\gf).\label{fe-4}
\end{align}
Thus, substituting \eqref{fe-1}-\eqref{fe-4} into \eqref{fe-0} we can quickly obtain \eqref{conj3}.
\end{proof}

\begin{re}\label{re:humanProof}
We provide here two other direct proofs of \eqref{fe-1}. First, by \cite[Remark 3.18]{Zhao2001} we have
$$
\begin{array} {rl}
{\displaystyle \phantom{\frac12} } \ & \Li_{2,1}(y,x)
={\displaystyle \Li_3^*(x)- \Li_3^*\Bigl(\frac{x-xy}{1-xy}\Bigr)+\Li_3^*(xy)
- \Li_3\Bigl(\frac{y-xy}{1-xy}\Bigr)}\\
 {\displaystyle \phantom{\frac12}}\ & \phantom{\Li_{1,2}(x,y)}
+ \Li_3(y) - \Li_3(xy) -\ln(1-xy) (\Li_2(x) + \Li_2(y)) \\
 {\displaystyle \phantom{\frac12}}\ & \phantom{\Li_{1,2}(x,y)}
{\displaystyle
-\frac12\ln^2\Bigl(\frac{1-x}{1-xy}\Bigr)\ln\Bigl(\frac{1-y}{1-xy}\Bigr),}
\end{array}$$
where
$$
\Li_3^*(x) = \Li_3(1)- \Li_3(1-x)+\Li_2(1)\ln(1-x)-
\frac12\ln(x)\ln^2(1-x).$$
Taking $y=\gf^2$ and using $1-y=\gf$ we see that
\begin{equation*}
\lim_{x\to 1^-} \Li_3^*(x)- \Li_3^*\Bigl(\frac{x-xy}{1-xy}\Bigr)= \Li_2(1)\ln(1-y)=\frac{\pi^2}{6}\ln(\gf).
\end{equation*}
Hence
\begin{align*}
\Li_{2,1}(\gf^2)=\Li_3(1)-\Li_3(\gf)-\ln(\gf)\Li_2(\gf^2)-\ln^3(\gf)+\frac{\pi^2}{6}\ln(\gf) =\ze(3)+\frac{\pi^2}{10}\ln(\gf)-\Li_3(\gf)
\end{align*}
by \eqref{fe-3}. This proves \eqref{fe-1}. Second, setting $k=r=2$ in \cite[Thm. 8]{AM1999} or \cite[Thm. 2.1]{X2020} yields
\begin{align*}
\Li_{2,1}(x)=\ze(3)-\Li_3(1-x)+\ln(1-x)\Li_2(1-x)+\frac1{2}\ln(x)\ln^2(1-x).
\end{align*}
Letting $x=\gf^2$ and noting the fact that (see \cite[(1.20)]{Lewin1981})
\[\Li_2(\gf)=\frac{\pi^2}{10}-\ln^2(\gf)\]
gives
\begin{align*}
\Li_{2,1}(\gf^2)=\ze(3)-\Li_3(\gf)+\ln(\gf)\Li_2(\gf)+\ln^3(\gf)=\ze(3)+\frac{\pi^2}{10}\ln(\gf)-\Li_3(\gf).
\end{align*}
A human proof of \eqref{fe-2} is possible although it is conceivably much more complicated than \eqref{fe-1}. In particular, setting $k=3$ and $r=2$ \cite[Thm. 2.1]{X2020} yields
\begin{align*}
\Li_{3,1}(x)+\Li_{3,1}(1-x)&=\ln(x)\left(\ze(3)-\Li_3(1-x)+\ln(1-x)\Li_2(1-x)\right)\\&\quad+\ze(3,1)+\ln(1-x)\Li_{2,1}(1-x)+\frac{1}{4}\ln^2(x)\ln^2(1-x).
\end{align*}
Letting $x=\gf$ gives
\begin{align}
\Li_{3,1}(\gf)+\Li_{3,1}(\gf^2)=\frac{\pi^4}{360}+\frac{\pi^2}{5}\ln^2(\gf)-\frac1{3}\ln^4(\gf)-2\ln(\gf)\Li_3(\gf)+\frac{11}{5}\ln(\gf)\ze(3).
\end{align}
\end{re}

\medskip

\noindent{\bf Acknowledgement.}  Ce Xu is supported by the National Natural Science Foundation of China (Grant No. 12101008), the Natural Science Foundation of Anhui Province (Grant No. 2108085QA01) and the University Natural Science Research Project of Anhui Province (Grant No. KJ2020A0057). Jianqiang Zhao is supported by the Jacobs Prize from The Bishop's School.

\medskip

\noindent{\bf Disclosure statement.} The authors report there are no competing interests to declare.


\begin{thebibliography}{99}

\bibitem{Akhilesh}
P.\ Akhilesh, Multiple zeta values and multiple Ap\'ery-like sums, \emph{J. Number Theory.} \textbf{226}(2021), pp.\ 72--138.

\bibitem{AM1999}
T. Arakawa and M. Kaneko, Multiple zeta values, poly-Bernoulli numbers, and related zeta functions, \emph{Nagoya Math. J.} \textbf{153} (1999), pp.\ 189--209.

\bibitem{Au2020}
K.C. Au, Evaluation of one-dimensional polylogarithmic integral, with applications to infinite series, arXiv:2007.03957.

\bibitem{Au2021}
K.C. Au, Iterated integrals and special values of multiple polylogarithm at algebraic arguments, arXiv:2201.01676.

\bibitem{Chu2020}
W. Chu, Alternating series of Ape\'ry-type for the Riemann zeta function, \emph{Contribut. Discrete Math.} \textbf{15}(2020), pp.\ 108--116.

\bibitem{Chu2021}
W. Chu, Further Ape\'ry-like series for Riemann zeta function, \emph{Math. Notes} \textbf{109}(2021), pp.\ 136--146.


\bibitem{DavydychevDe2004}
A.I.\ Davydychev and M.\ Yu.\ Kalmykov, Massive Feynman diagrams and inverse binomial sums, \emph{Nuclear Phys.\ B}
\textbf{699} (2004), pp.\ 3--64. arXiv:hep-th/0303162v4.

\bibitem{Hoffman2000}
M.E. Hoffman, Quasi-shuffle products, \emph{J. Algebraic Combin.} \textbf{11}(2000), 49--68.

\bibitem{HP2012}
Kh. Hessami Pilehrood and T. Hessami Pilehrood, Congruences arising from Ape\'ry-type series for zeta values, \emph{Adv. Appl. Math.} \textbf{49}(2012), 218-238.

\bibitem{Lewin1981}
L.\ Lewin, \emph{Polylogarithms and Associated Functions}, Elsevier Sci.\ Publishers,
New York, New York, 1981.

\bibitem{Sun2015}
Z.-W.\ Sun, New series for some special values of $L$-functions, \emph{Nanjing Univ. J. Math. Biquarterly} \textbf{32}(2015), no.2, 189-218.

\bibitem{Sun2021}
Z.-W.\ Sun, \emph{New Conjectures in Number Theory and Combinatorics (in Chinese)}, Harbin Institute of Technology Press, Harbin, 2021.

\bibitem{X2020}
C. Xu, Explicit relations between multiple zeta values and related variants, \emph{Adv. Appl. Math.} {\bf 130}(2021), 102245.

\bibitem{Zhao2001}
J.\ Zhao, Motivic complexes of weight three and pairs of simplices in
projective 3-space, \emph{Adv.\ Math.} \textbf{161}(2001), pp.\ 141--208.

\bibitem{Zhao2016}
J. Zhao, \emph{Multiple Zeta Functions, Multiple Polylogarithms and Their Special Values}, Series on Number
Theory and its Applications, Vol.~12, World Scientific Publishing Co. Pte. Ltd., Hackensack, NJ, 2016.

\end{thebibliography}
\end{document}